\documentclass{article}

\usepackage{amsmath, amsfonts, amsthm, amssymb}
\usepackage{mathtools}

\usepackage{microtype}

\let\OLDthebibliography\thebibliography
\renewcommand\thebibliography[1]{
	\OLDthebibliography{#1}
	\setlength{\parskip}{1pt}
	\setlength{\itemsep}{1pt plus 0.3ex}
}

\def\Z{\mathbb{Z}}
\def\R{\mathbb{R}}
\def\F{\mathbb{F}}

\newtheorem{theorem}{Theorem}[section]
\newtheorem{lemma}[theorem]{Lemma}
\newtheorem{proposition}[theorem]{Proposition}
\newtheorem{corollary}[theorem]{Corollary}

\theoremstyle{definition}
\newtheorem{problem}[theorem]{Problem}
\newtheorem{definition}[theorem]{Definition}
\newtheorem{remark}[theorem]{Remark}

\def\int{\operatorname{Int}}
\def\es{\operatorname{S}}
\def\de{\operatorname{D}}
\def\reeb#1{\mathcal{R}(#1)}
\def\ind{\operatorname{ind}}
\def\degin{\deg_{in}}
\def\degout{\deg_{out}}

\usepackage{tikz}

\makeatletter
\def\@addpunct#1{%
	\relax\ifhmode
	\ifnum\spacefactor>\@m \else#1\fi
	\fi}
\newcommand{\keywordsname}{$2010$ Mathematics Subject Classification}
\def\@setkeywords{%
	{\itshape \keywordsname.}\enspace \@keywords\@addpunct.}
\def\keywords#1{\def\@keywords{#1}}
\let\@keywords=\@empty
\g@addto@macro{\maketitle}{\begingroup%
	\let\@makefnmark\relax  \let\@thefnmark\relax%
	\ifx\@keywords\@mpty\else\@footnotetext{\@setkeywords}\fi%
	\endgroup}
\makeatletter

\keywords{Primary: 58K05; Secondary: 57M15, 58K65.  \indent\indent The author was supported by the Polish Research Grant NCN 2015/19/B/ST1/01458}

\newcommand{\address}{{ \bigskip

\footnotesize
	{\noindent\textsc{\L{}ukasz Patryk Michalak}\\
		Adam Mickiewicz University in Pozna\'n\\
		Faculty of Mathematics and Computer Science\\
		ul. Umultowska 87, 61-614 Pozna\'n, Poland} \\
		\textit{E-mail address:} \texttt{lukasz.michalak@amu.edu.pl}
		
}}

\date{}

\title{Realization of a~graph as~the~Reeb graph of~a~Morse~function on~a~manifold}

\author{\L{}ukasz Patryk Michalak
}

\begin{document}

	\maketitle
	\begin{abstract}
		We investigate the problem of the realization of a given graph as~the Reeb graph $\reeb{f}$ of a smooth function $f\colon M\rightarrow \R$ with finitely many critical points, where $M$ is a closed manifold. We show that for any $n\geq2$ and any graph $\Gamma$ admitting the so-called good orientation there exist an $n$-manifold~$M$ and a Morse function $f\colon M\rightarrow \R $ such that its Reeb graph $\reeb{f}$ is isomorphic to $\Gamma$, extending previous results of Sharko and Masumoto--Saeki. We prove that Reeb graphs of simple Morse functions maximize the number of cycles. Furthermore, we provide a complete characterization of graphs which can arise as Reeb graphs of surfaces.
	\end{abstract}

	\section{Introduction} 		     
	
	The Reeb graph, denoted by $\reeb{f}$, has been introduced by Reeb \cite{Reeb}. It is defined for a closed manifold~$M$ and a~smooth function $f\colon M\rightarrow \R$ with finitely many critical points by contracting the connected components of level sets of~$f$. Reeb graphs are widely applied in computer graphics and shape analysis (see~\cite{applications}).

	Sharko \cite{Sharko} and Masumoto--Saeki \cite{Saeki} proved that every graph which admits a good orientation (see Definition \ref{def:good_orientation}) can be realized as the Reeb graph of a~function with finitely many critical points on a surface. This leads to a natural problem:
	
	\begin{problem}\label{main_problem} For a given manifold $M$, which graph $\Gamma$ can be realized as the Reeb graph of a function $f\colon M \rightarrow \R$ with finitely many critical points?
	\end{problem}
	
	We give a complete answer to this problem for surfaces (Theorem \ref{thm:reeb_graphs_of_surfaces}). Let $\Gamma_0$ be the connected graph with two vertices and one edge. A graph $\Gamma \neq \Gamma_0$ can be realized as the Reeb graph of a function with finitely many critical points	if and only if it admits a good orientation and the number of cycles in $\Gamma$ (i.e.~its first Betti number) is not greater than the maximal attainable. Therefore the realizability of an oriented graph depends only on its homotopy type and whether its orientation is a good orientation.
	Next we describe graphs which can be realized by Morse functions (Theorem \ref{theorem:realization_by_Morse_function}) on a given surface. It depends additionally on the number of vertices of degree $2$ in a graph.
	We also extend the theorem of Sharko and Masumoto--Saeki (Theorem \ref{thm:graph_as_reeb_graph}) by constructing, for any given $n\geq 2$, an $n$-manifold and a Morse function such that its Reeb graph is isomorphic to a given graph $\Gamma$ with good orientation.
	The proofs of these theorems uses Morse theory and handle decomposition of manifolds.
	
	In the construction we need the information on $\reeb{M}$, which is  defined as the maximal number of cycles in Reeb graphs of functions with finitely many critical points on a manifold $M$. We call it the Reeb number of $M$. There are a~few ways to calculate this number for surfaces. One can do it using techniques of Kaluba--Marzantowicz--Silva  \cite{KMS} for orientable surfaces or using the notion of co-rank of the fundamental group (Gelbukh \cite{Gelbukh:Reeb_graph,Gelbukh:corank}). Our calculation relies on results of Cole-McLaughlin et al. \cite{Edelsbrunner} describing the number of cycles in the Reeb graph of a simple Morse function on a surface. We show that the Reeb number is attained by simple Morse functions (Lemma~\ref{lemma:maximalization_by_simple_Morse_functions}) what implies that $\reeb{\Sigma} = \left\lfloor \frac{k}{2} \right\rfloor$ for a closed surface $\Sigma$ of the Euler characteristic $\chi(\Sigma) = 2-k$.
	
	The paper is organized as follows. In Section \ref{section:basic_properties_of_Reeb_graphs} we recall basic properties of Reeb graphs. Section \ref{section:number_of_cycles_in_Reeb_graph} is devoted to the computation of the Reeb number of surfaces. We give a simple proof of lemmas from \cite{Edelsbrunner} and provide a proof that simple Morse functions maximize the number of cycles in Reeb graphs. In~Section \ref{section:Realization_by_Morse_function} we prove an analogue to the Sharko and Masumoto--Saeki theorem in
	arbitrary dimension. Finally, in Section \ref{section:main_problem_for_surfaces} we resolve Problem \ref{main_problem} for surfaces.
	
	\section{Basic properties of Reeb graphs}\label{section:basic_properties_of_Reeb_graphs}
	
	Throughout the paper we assume that all manifolds are compact, smooth, connected of
	dimension $n \geq 2$ and that all graphs are finite and connected.
	
	\begin{definition}
		Let $X$ be a topological space and $f\colon X\rightarrow \R$ be a function. We say that $x$ and $y$ are in \textbf{Reeb relation} $\sim_{\mathcal{R}}$ if and only if $x$ and $y$ belong to the same connected component of a level set of $f$. The quotient space $X/_{\sim_{\mathcal{R}}}$ is denoted by $\reeb{f}$. 
	\end{definition}
	
	A smooth triad is a tripple $(W,W_-,W_+)$, where $W$ is a compact manifold with boundary $\partial W = W_- \sqcup W_+$, the disjoint union of $W_-$ and $W_+$ (we allow $W_\pm = \varnothing$). A smooth function on a triad  $(W,W_-,W_+)$ is a smooth function $f\colon W \rightarrow [a,b]$ such that $f^{-1}(a) = W_-$, $f^{-1}(b) = W_+$ and  $f$ has all critical points in $\int W$, the interior of $W$.
	
	It is known that for a smooth function $f$ with finitely many critical points on a smooth triad $(W,W_-,W_+)$, the quotient space $\reeb{f}$ is homeomorphic to a~finite graph, i.e. to~a~one-dimensional finite CW-complex (see \cite{Reeb}, \cite{Sharko}), known as the \textbf{Reeb graph}. The vertices of $\reeb{f}$ correspond to the critical components of the level sets of $f$ (i.e. to the components containing critical points) and to the components of $W_\pm$. We consider the combinatorial structure on $\reeb{f}$ induced by the homeomorphism as the standard one.

	Every function $f$ induces the continuous function $\overline{f}\colon \reeb{f} \rightarrow \R$ which satisfies $f = \overline{f} \circ q$, where $q\colon W \rightarrow W/_{\sim_\mathcal{R}} = \reeb{f}$ is the quotient map. The function $\overline{f}$ gives an orientation to the edges of $\reeb{f}$ as described by Sharko \cite{Sharko} and is called the good orientation of a graph by Masumoto--Saeki~\cite{Saeki}.
	
	For a~vertex $v$ in an oriented finite graph $\Gamma$ let $\degin(v)$ ($\degout(v)$) denote the number of incoming (outcoming) edges to the vertex $v$.	The \textbf{degree} of $v$ is defined as $\deg(v) = \degin(v) + \degout(v)$.
	
	
	\begin{definition}[\cite{Sharko}]\label{def:good_orientation}
		A \textbf{good orientation} of a graph $\Gamma$ is the orientation induced by a continuous function $g\colon \Gamma \rightarrow \R$ such that $g$ is strictly monotonic on the edges and has extrema in the vertices of degree~one.
	\end{definition}

	Note that not all graphs admit a good orientation (see example \cite[Figure~1]{Sharko}). It is easy to see that the orientation on $\reeb{f}$ induced by $\overline{f}$ satisfies the condition in the definition.
	
	We say that an oriented graph $\Gamma$ \textbf{can be realized} as a Reeb graph for a~closed manifold $M$ if there exists a function $f$ on $M$ with finitely many critical points such that $\reeb{f}$ is isomorphic to $\Gamma$ as an oriented graph. Therefore a good orientability is the necessary condition for $\Gamma$ in Problem~\ref{main_problem}.
	
	Throughout the paper isomorphism of graphs denotes isomorphism of oriented graphs, where applicable.
	
	\section{Number of cycles in Reeb graphs}\label{section:number_of_cycles_in_Reeb_graph}
	
	Recall that for a CW-complex $X$ the first Betti number $\beta_1(X)$ is the rank over $\Z$ of the first homology group, i.e. $\beta_1(X) := \operatorname{rank}_\Z \operatorname{H}_1(X;\Z)$. For a finite graph $\Gamma$ its first Betti number $\beta_1(\Gamma)$ is called the \textbf{number of cycles} in $\Gamma$.

	We write $\Gamma = \Gamma(V,E)$ when $V$ is the set of vertices and $E$ is the set of edges of a graph $\Gamma$. Then the number of cycles in the graph $\Gamma$ is equal to $\beta_1(\Gamma)= |E|-|V|+1$.
	
	\begin{definition}\label{def:reeb_number}
		The \textbf{Reeb number} $\reeb{M}$ of a manifold $M$ is the maximal number of cycles among all Reeb graphs of smooth functions on $M$ with finitely many critical points, i.e.
		\[
			\reeb{M} :=\max \left\{\, \beta_1(\reeb{f})\, |\, f\colon M \rightarrow \R \text{ has finitely many critical points}\, \right\}\!.
		\]
	\end{definition}
	
	From \cite{Edelsbrunner} and \cite{KMS} we know that $\beta_1(\reeb{f}) \leq \beta_1(M)$ for every smooth function $f\colon M \rightarrow \R$ with finitely many critical points, hence $\reeb{M}$ is well-defined and $\reeb{M} \leq \beta_1(M)$.
	
	A smooth function $f$ on a triad $(W,W_-,W_+)$ is a Morse function if all its critical points are non-degenerate. A Morse function is \textbf{simple} if each critical value corresponds exactly to one critical point. For such a function the number of vertices in its Reeb graph is the same as the number of critical points.
	
	Let $\Sigma_g := \#_{i=1}^g T^2$ denote the closed orientable surface of genus $g$ which is the connected sum of $g$ copies of the two-dimensional torus $T^2$ and let $S_g := \#_{i=1}^g \R P^2$ denote the closed non-orientable surface of genus $g$, the connected sum of $g$ copies of the real projective plane $\R P^2$.
	
	Cole-McLaughlin, Edelsbrunner et al. \cite{Edelsbrunner} established the following lemma.
	
	\begin{lemma}[{\cite[Lemma A and C]{Edelsbrunner}}]\label{lemma:cycles_for_simple_Morse_function_on_surface}
		
		Let $f\colon \Sigma \rightarrow \R$ be a simple Morse function on a closed surface $\Sigma$.
		\begin{enumerate}
			\item If $\Sigma = \Sigma_g$, then $\beta_1(\reeb{f}) = g$.
			\item If $\Sigma = S_g$, then $\beta_1(\reeb{f}) \leq \left\lfloor \frac{g}{2} \right\rfloor$, where $\left\lfloor x\right\rfloor$ is the floor of $x$.
		\end{enumerate}
	\end{lemma}
	
	Their proof in the orientable case follows by collapsing vertices of degree 1 in $\reeb{f}$, merging arcs across vertices of degree $2$ and a simple calculation.
	In the non-orientable case the composition of a particular orientable $2$-sheeted covering map and $f$ is perturbed to obtain a simpe Morse function. They state that it does not alter the Reeb graph. While the argument may be correct in their situation, it is worth pointing out that it is not the case for perturbations of arbitrary Morse functions.
	
	Consider a Morse function $f \colon \Sigma_g \rightarrow \R$ for $g \geq 1$ which is  \textbf{self-indexing}, i.e.~for every critical point $p$ we have $f(p) = \ind(p)$, where $\ind(p)$ is the index of critical point $p$ (see \cite[Alternate version of 4.8.]{Milnor}). The function has only three critical levels 0, 1 and 2 and all the vertices of $\reeb{f}$ are on these levels. Since the vertices on levels 0 and 2 have degree 1 and the edges join vertices from different levels, the Reeb graph $\reeb{f}$ is a tree. However, by Lemma \ref{lemma:cycles_for_simple_Morse_function_on_surface}, any perturbation of $f$ which is a simple Morse function has in its Reeb graph exactly $g \geq 1$ cycles.

	We give another elementary proof of Lemma \ref{lemma:cycles_for_simple_Morse_function_on_surface}. We need the following fact which is also used in \cite{Edelsbrunner}.
	
	\begin{proposition}\label{proposition:degree_and_index}
		Let $f \colon \Sigma \rightarrow \R$ be a simple Morse function on a closed surface. Let $p$ be a critical point of the function and $v := q(p)$ be the vertex of $\reeb{f}$ corresponding to $p$. Then $\ind(p) = 0$ or~$2$ if and only if $\deg(v)=1$; $\ind(p) =1$ if and only if
		\begin{eqnarray*}
			\deg (v) =  \begin{cases}
				3	        &\text{\ \ \ \ if $\Sigma$ is orientable,}\\
				2\text{ or }3	        &\text{\ \ \ \ if $\Sigma$ is non-orientable.}
			\end{cases}
		\end{eqnarray*}
	\end{proposition}

	Recall that for an $n$-manifold $W$ with non-empty boundary an embedding $\varphi\colon\es^{k-1} \times \de^{n-k} \to \partial W$ is \textbf{non-trivial}, if $[\varphi_{|_{\es^{k-1}\times\{0\}}}] \in \pi_{k-1}(\partial W)$ is a non-trivial element of the $(k-1)$th homotopy group of $\partial W$ (taking basepoints as appropriate). In the case $k=1$ the non-triviality means that the image of $\varphi$ belongs to separate connected components of $\partial W$.

	\begin{proof}
		Critical points of index 0 or 2 are extrema of $f$, so the degrees of the corresponding vertices are equal to 1.
		
		For $\ind(p)=1$ the numbers of edges incoming to $v$ and outcoming from $v$ are related (by Morse theory) to the number of the connected components of sub-level surface after attaching a 1-handle $\de^1 \times \de^1$ using an embedding $\varphi \colon \es^0 \times \de^1 \rightarrow \partial W$, where $W = f^{-1}\left((-\infty,f(p)-\varepsilon]\right)$ for some sufficiently small $\varepsilon >0$. 
		
		If $\varphi$ is a non-trivial embedding, i.e. the two copies of $\de^1$ belong to separate boundary components, in the result these components merge, so $\deg(v) =3$. If~the embedding is trivial and preserves the orientation, the component splits into two, so $\deg(v) =3$. However, if $\varphi$ is not orientation-preserving, then we obtain one connected component, so $\deg(v)=2$.	
	\end{proof}

\begin{remark}
	Note that following the original paper of Reeb \cite[Th\'{e}or\`{e}me 3]{Reeb}, some authors (e.g. Biasotti et al. in [1]) overestimate the bound on the degree of a vertex	in the Reeb graph of a simple Morse function on a non-orientable surface. The straightforward conclusion from the Morse Lemma bounds the~degree by $4$, but the estimate given by Proposition \ref{proposition:degree_and_index} is sharp.
\end{remark}
	
	We denote by $\chi(M)$ the Euler characteristic of a closed $n$-manifold $M$. Let $f\colon M \rightarrow \R$ be a Morse function with $k_i$ critical points of index $i$. It is a standard fact that $\chi(M) = \sum_{i=0}^n(-1)^i k_i$.
	
	\begin{proof}[Proof of Lemma \ref{lemma:cycles_for_simple_Morse_function_on_surface}]
		Let $V$ and $E$ be the sets of vertices and edges of $\reeb{f}$, respectively. Since $f$ is a simple Morse function, $\chi(\Sigma) = k_0 - k_1 + k_2$ and $|V|= k_0+k_1+k_2$. By Euler's handshaking lemma $2|E| = \sum_{v\in V}\deg(v)$.
		
		The above proposition yields $2|E| = k_0+3k_1+k_2$ if $\Sigma = \Sigma_g$ is orientable. Then
		\begin{eqnarray*}
			2\beta_1(\reeb{f}) &=& 2\left(|E| - |V| +1\right) \\&=& k_0 +k_2 + 3k_1 - 2\left(k_0+k_1+k_2\right) + 2 
			\\&=& 2 - \left(k_0 -k_1+k_2\right) = 2 - \chi(\Sigma_g) = 2g.
		\end{eqnarray*}
		However, if $\Sigma = S_g$ is a non-orientable surface, the above proposition gives only $2|E| \leq k_0 +3k_1 + k_2$ and so $2\beta_1(\reeb{f}) \leq2 - \chi(S_g) = g$.
	\end{proof}

	\paragraph{ Maximizing the number of cycles.} Now we show that Reeb graphs of simple Morse functions maximize the number of cycles.

	\begin{lemma} 
		\label{lemma:maximalization_by_simple_Morse_functions}
		Let $f\colon M \rightarrow \R$ be a function with finitely many critical points on a closed manifold $M$. Then there exists a simple Morse function $g\colon M \rightarrow \R$ such that $\beta_1(\reeb{g}) \geq \beta_1(\reeb{f})$.
	\end{lemma}
	
	\begin{proof}
		Let $v$ be a vertex of $\reeb{f}$, $c:= \overline{f}(v)$ and let $V$ be a~small closed contractible neighbourhood of $v$
		(as in Figure \ref{rys:neighbourhood_of_vertex}), 
		such that $W:= q^{-1}(V)$ is~a~compact submanifold of $M$ with boundary $\partial W = W_- \sqcup W_+$, where $W_\pm$ is a~submanifold of $f^{-1}(c\pm\varepsilon)$ for some sufficiently small $\varepsilon >0$. By \cite[Lemma 2.8.]{Milnor} there exists a simple Morse function $h\colon W \rightarrow [c-\varepsilon,c+\varepsilon]$ on the smooth triad $(W,W_-,W_+)$ which has critical values different from critical values of  $f|_{M\setminus W}$.
		
		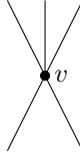
\begin{figure}[h]
			\centering
			
			\begin{tikzpicture}[scale=1]

			\filldraw (0,0) circle (1.7pt)   node[align=left, right] {$v$};

			\draw (0,0) -- (0.5,1);
			\draw (0,0) -- (0,1);
			\draw (0,0) -- (-0.5,1);
			\draw (0,0) -- (-0.5,-1);
			\draw (0,0) -- (0.5,-1);

			\end{tikzpicture}
			
			\caption{\small The small neighbourhood $V$ of the vertex $v$.}
			\label{rys:neighbourhood_of_vertex}
		\end{figure}
		
		Let us define the function $g\colon M \rightarrow \R$ by
		\begin{eqnarray*}
			g (x) =  \begin{cases}
				h(x)	        &\text{\ \ \ \ if $x \in W\!$,}\\
				f(x)	        &\text{\ \ \ \ if $x \in M \setminus \int W\!$.}
			\end{cases}
		\end{eqnarray*}
	
		Clearly, $g$ is smooth, with the same critical points as $f$ and $h$. Since $V$ is contractible we have homotopy equivalence $\reeb{f} \simeq \reeb{f}/ V$, which is in turn equivalent to $\reeb{g} / \reeb{h}$. Let $T$ be a spanning tree of the graph $\reeb{h}$. It is easily seen that $\reeb{g}/T$ is homeomorphic to $\reeb{g}/\reeb{h}$ after adding the remaining edges from $\reeb{h}/T$. Hence $$\beta_1(\reeb{g}) = \beta_1(\reeb{g}/T) \geq \beta_1(\reeb{g}/\reeb{h}) = \beta_1(\reeb{f}).$$
		
		After performing the procedure as described above for each vertex $v$ of $\reeb{f}$ we obtain the desired simple Morse function.
	\end{proof}

	\begin{remark}
		For a Morse function $f$, the simple Morse function $g$ in the above lemma can be obtained without changing the critical points of $f$ and their indices. Moreover, $g$ may differ from $f$ only in small neighbourhoods of the critical points as in \cite[Lemma 2.8.]{Milnor}.
	\end{remark}
	
	\begin{corollary}\label{corollary:reeb_number_from_simple_Morse_functions}
		$$\reeb{M} =\max \left\{\, \beta_1(\reeb{f})\, |\, f\colon M \rightarrow \R \text{ is a simple Morse function}\, \right\}\!.$$
	\end{corollary}
	
	By the above corollary and Lemma \ref{lemma:cycles_for_simple_Morse_function_on_surface} we conclude that $\reeb{\Sigma_g} = g$ and $\reeb{S_g} \leq \left\lfloor \frac{g}{2} \right\rfloor$. To obtain equality in the case of non-orientable surfaces we construct an explicit example of a function which realizes maximal number of cycles in Theorem \ref{thm:reeb_graphs_of_surfaces}. See also \cite{Edelsbrunner} for an alternate description.
	
	\begin{corollary}\label{corollary:Reeb_number_for_surfaces}
		$\reeb{\Sigma_g} = g$ and $\reeb{S_g} = \left\lfloor \frac{g}{2} \right\rfloor$. In other words, if a~closed surface $\Sigma$ has the Euler characteristic $\chi(\Sigma) = 2-k$, then $\reeb{\Sigma} = \left\lfloor \frac{k}{2} \right\rfloor$. \qed
	\end{corollary}

	The upper bound $\reeb{\Sigma_g} \leq g$ can also be found in \cite[Theorem 5.6.]{KMS}, where the proof relies on the maximal number of non-separating circle embeddings in an orientable surface.

	For an orientable, closed manifold $M$ Gelbukh \cite[Theorem 13]{Gelbukh:Reeb_graph} proved that the Reeb number $\reeb{M}$ is equal to the co-rank of the fundamental group $\pi_1(M)$ (i.e. the maximal number $r$ for which there exists epimorphism $\pi_1(M)\to \F_r$ onto the free group of rank $r$). She also showed that $\reeb{M}$ can be attained by the Reeb graph of a~simple Morse function.
		
	\section{Realization by a Morse function}\label{section:Realization_by_Morse_function}
	
	\begin{lemma}\label{lemma:realization_of_neighbourhood_of_vertex_with_spheres_on_the_boundary}
		Let $n \geq 2$, $k_-,k_+ \geq 1$ and $t \geq 0$ ($k_-$ or $k_+ \geq 2$ for $t=0$) be integers and let $c\in (a,b)$. Then there exist a compact $n$-manifold $N$ with boundary $\partial N= N_- \sqcup N_+$, where $N_\pm := \bigsqcup_{i=1}^{k_\pm} \es^{n-1}$ is the disjoint union of $(n-1)$-dimensional spheres with the standard smooth structure, and a~Morse function $f\colon N \rightarrow [a,b]$ on the smooth triad $(N,N_-,N_+)$ with only one critical value $c$ and $(k_++t-1)$ critical points of index $(n-1)$ and $(k_-+t-1)$ critical points of index $1$.
	\end{lemma}
	
		We will use the correspondence between attaching $k$-handles and  critical points of index $k$ \cite[Theorem 3.12.]{Milnor}. From \cite[Lemma 2.1.]{Kervaire} we know that the sphere $\es^{n-1}$ with the standard smooth structure serves as identity element of the~connected sum operation. As the result of attaching a 1-handle along a non-trivial	embedding $\es^0 \times \de^{n-1} \to \partial W \cong \es^{n-1} \sqcup \es^{n-1}	\sqcup P$ with the image contained in $\es^{n-1} \sqcup \es^{n-1}$ we obtain a new manifold $W'$ with $\partial W' \cong (\es^{n-1}\#\es^{n-1})
		\sqcup P \cong \es^{n-1} \sqcup P$. The operation thus reduces the number of spheres in the boundary of~$W$. The dual operation to the above (attachment of $(n-1)$-handle via trivial embedding) increases the number of connected components --- one copy of $S^{n-1}$ splits into two copies of $\es^{n-1}$.
		
		\begin{proof}
		Let $N_\pm$ be as in the statement. We start with $N_- \times [a,a+\varepsilon]$ and attach $(k_++t-1)$ $(n-1)$-handles along trivial embeddings with the correct orientation to $N_-\times \{1\}$, as described above. This produces a manifold $W$ with $\partial W = N_- \sqcup P$, where $P$ is the disjoint union of $(k_- + k_++t-1)$ copies of $\es^{n-1}$. On $N_- \times [a,a+\varepsilon]$ the projection on the second factor is a Morse function without critical points, so attaching the handles corresponds (by \cite[Theorem 3.12.]{Milnor}) to a Morse function $h\colon W \rightarrow  [a,a+2\varepsilon]$ with $(k_++t-1)$ critical points all of index $(n-1)$.
		 
		  Next we attach $(k_-+t -1)$ 1-handles to $W$ along non-trivial embeddings to $P$ to obtain a connected manifold $N$ with boundary $\partial N = N_- \sqcup N_+$ (the manifold $W$ has $k_-$ connected components). Hence there exists a Morse function $h'\colon N \rightarrow [a,b]$ with $(k_++t-1)$ critical points of index $(n-1)$ and $(k_-+t-1)$ critical points of index~1.
		
		Since we attached $(n-1)$-handles before 1-handles, the critical points of the function $h'$ of index $(n-1)$ are below the critical points of index 1. Therefore we can apply the method from \cite[Theorem 4.1.; Extension 4.2.; Theorem 4.4.]{Milnor} to obtain a new Morse function $f\colon N \rightarrow [a,b]$ on the smooth triad $(N,N_-,N_+)$ with all criticals points moved to one critical level~$c$. 
	\end{proof}
	
	\begin{remark}\label{remark:additional_non_orientable_1-handles}
		In the same way as in the above lemma we can obtain a non-orientable surface $N$ (for $n=2$ and $t=0$) with a Morse function with a unique critical value and $ k_- + k_+ +r -2$ critical points of index 1, by attaching $r$ additional 1-handles along trivial embeddings with non-compatible orientations.
	\end{remark}
	
	Sharko \cite[Theorem 2.1.]{Sharko} proved that for a graph $\Gamma$ with good orientation there exist a surface $\Sigma$ and a function $f\colon \Sigma \rightarrow \R$ with finitely many critical points such that $\reeb{f}$ is isomorphic to~$\Gamma$. In general, $f$ may have degenerate critical points, so it is not a Morse function. Sharko stated (without a proof) that methods of Takens in \cite{Takens} lead to a version of this theorem for $n$-dimensional manifolds. We will give an explicit construction, in terms of handle decomposition, of such a manifold $M$ of an arbitrary given dimension $n\geq 2$ and a Morse function on $M$ such that its Reeb graph is isomorphic to $\Gamma$.

	\begin{theorem}\label{thm:graph_as_reeb_graph}
		Let $\Gamma$ be a finite graph with good orientation and $n\geq 2$. Then there exist a smooth closed $n$-manifold $M$ and a Morse function $f\colon M \rightarrow \R$ such that the Reeb graph $\reeb{f}$ is isomorphic to $\Gamma$.
	\end{theorem}
	
	\begin{proof}
		Let $g\colon \Gamma \rightarrow \R$ be a function inducing the good orientation of $\Gamma$. Since $\Gamma$ is finite we can define the positive real number $$\varepsilon := \frac{1}{3}\min \left\{ |g(v) - g(w)| : \text{$v$ and $w$ are ends of an edge of $\Gamma$} \right\}\!.$$
				
		Let $v$ be a vertex of $\Gamma$ of degree one. Set $N_v := \de^n$ and define $f\colon N_v \to [g(v) - \varepsilon, g(v) + \varepsilon]$ by $f_v(x_1,\ldots,x_n) = g(v) \pm \varepsilon( x_1^2 + \ldots + x_n^2)$, where sign $\pm$ depends whether  $v$ is minimum or maximum of $g$. Obviously $f_v$ has one critical point in the center of the disc.
		
		The function $g$ has extrema only in vertices of degree 1, hence for each vertex $v$ of degree at least two we have inequalities $k_- := \degin(v) \geq 1$ and $k_+ := \degout(v) \geq 1$. By Lemma \ref{lemma:realization_of_neighbourhood_of_vertex_with_spheres_on_the_boundary} 
		(for $t=1$) we obtain an $n$-manifold $N_v$ with boundary $\partial N_v = N_v^- \sqcup N_v^+$, where $N_v^\pm := \bigsqcup_{i=1}^{k_\pm} \es^{n-1}$ and a Morse function $f_v\colon N_v \rightarrow [g(v) -\varepsilon, g(v) + \varepsilon]$ such that $f_v^{-1}\left(g(v)\pm \varepsilon\right) = N_v^\pm$ and the only critical value of $f_v$ is $g(v)$ 
		(it has $k_-+k_+ +2t -2 \geq 2$ critical points). 
		
		 In both cases, since $f_v$ has one critical level and $\reeb{f_v}$ is connected (as the image of the connected manifold $N_v$), it has one internal vertex corresponding to one critical value of $f_v$ and has $k_\pm$ ends of edges on the level $g(v)\pm \varepsilon$. Therefore $\reeb{f_v}$ is homeomorphic to a small neighbourhood of $v$ in $\Gamma$.
		
		For each oriented edge $e$ in $\Gamma$ from a vertex $v$ to $w$ we form the manifold $N_e = \es^{n-1} \times [g(v) + \varepsilon, g(w)-\varepsilon]$ and the Morse function $f_e\colon N_e \rightarrow [g(v) + \varepsilon, g(w)-\varepsilon]$, by the projection on the second factor.
		
		We glue smoothly the manifolds $N_v$ and $N_e$ using a diffeomorphism between (the corresponding to $e$) component $\es^{n-1}$ of $N_v^+$ and $\es^{n-1} \times \{g(v) + \varepsilon\}$. Similarly, we glue $N_w$ and $N_e$.
		On the resulting smooth manifold we define a function which is the piecewise extension of $f_v$, $f_e$ and $f_w$ and has the same critical points.
		
		Performing the operation as described above for each edge of $\Gamma$ we obtain a~smooth closed $n$-manifold $M$ and a Morse function $f\colon M \rightarrow \R$. It is clear from the construction that the Reeb graph $\reeb{f}$ is isomorphic to $\Gamma$. 
	\end{proof}
	
	\section{Solution of the main problem for surfaces}\label{section:main_problem_for_surfaces}

	\begin{corollary}\label{corollary:euler_char_for_N_v}
		For $n=2$ the surface $N$ obtained in Lemma \ref{lemma:realization_of_neighbourhood_of_vertex_with_spheres_on_the_boundary} has the Euler characteristic $\chi(N) = 2 - (k_- + k_++2t)$.
	\end{corollary}
	
	\begin{proof}
		Recall that $N$ is formed from $W = \bigsqcup_{i=1}^{k_-} \es^{1} \times [0,1]$, $(k_++t-1)$ handles of index $n-1=1$ and $(k_-+t-1)$ handles of index 1, i.e. $r:= k_-+k_++2t-2$ handles of index 1 in total. We have $\chi(W) = 0$, $\chi(\de^1 \times \de^1)=1$ and $\chi(\es^0 \times \de^1) = 2$. Since the 1-handles $\de^1 \times \de^1$ are attached along $\es^0 \times \de^1$ we obtain $\chi(N) = \chi(W) + r\cdot \chi(\de^1 \times \de^1) - r\cdot\chi(\es^0\times \de^1)
		= -r = 2 - (k_- + k_++2t)$. 
	\end{proof}
	
	We need to modify the construction of the surface in the proof of Theorem~\ref{thm:graph_as_reeb_graph} (for $n=2$) to obtain desirable properties.
	
	\begin{corollary}\label{corollary:surface_from_theorem}
		Let $\Gamma = \Gamma(V,E)$ be a graph with good orientation and let $g:=\beta_1(\Gamma)$. Then there exist a closed surface $\Sigma$ and a function $f\colon \Sigma \to \R$ with finitely many critical points such that $\reeb{f}$ is isomorphic to $\Gamma$ and $\Sigma$ can be taken to be either orientable of genus $g$ or non-orientable of genus $2g$ (for $g\geq 1$). If $\Gamma$ is a tree, then $\Sigma$ is diffeomorphic to $\es^2$.
	\end{corollary}
	\begin{proof}
		First, in the construction in the proof of Theorem~\ref{thm:graph_as_reeb_graph} (for $n=2$) we change the surfaces $N_v$ and functions $f_v$ by using $t=0$ in Lemma \ref{lemma:realization_of_neighbourhood_of_vertex_with_spheres_on_the_boundary}. Moreover, for each vertex $w$ of $\Gamma$ such that $\degin(w) = \degout(w) = 1$ let $N_w$ be $\es^1 \times [-\varepsilon,\varepsilon]$. We take a function $f_w\colon N_w \to [g(w) -\varepsilon, g(w) + \varepsilon]$ on the triad $(N_w, \es^1\times\{-\varepsilon\},\es^{1}\times\{\varepsilon\})$ with one (degenerate) critical point, which in local coordinates can be written as $f_w(x,y) = -x^2 +y^3 +g(w)$. Its Reeb
		graph is homeomorphic to a small neighbourhood of $w$ in $\Gamma$. Let $\Sigma$ be the resulting closed surface.
		
		By the previous corollary each surface $N_v$ has the Euler characteristic $\chi(N_v) = 2 -(\degin(v) + \degout(v))= 2 - \deg(v)$ and each $N_e$ has $\chi(N_e) = 0$, where $v\in V$ and $e\in E$. To obtain $\Sigma$ we attach $N_e$ along copies of $\es^1$, hence
		\begin{eqnarray*}
			\chi(\Sigma) &=& \sum_{v\in V} \chi(N_v) =\sum_{v\in V} \left(2 - \deg(v)\right) = 2|V| - \sum_{v\in V} \deg(v) \\&=& 2|V| - 2|E| = 2 - 2\left(|E|-|V|+1\right) = 2 - 2g,
		\end{eqnarray*}
		where $g = \beta_1(\Gamma)$.
		
		If $\Gamma$ is a tree, then $g=0$, so $\chi(\Sigma) = 2$ and, in consequence, $\Sigma \cong \es^2$. For $g\geq 1$ it is clear that the constructed surface in Lemma \ref{lemma:realization_of_neighbourhood_of_vertex_with_spheres_on_the_boundary} (for $n=2$) is orientable, so $\Sigma$ can be orientable or not, what depends on the way of attaching $N_e$ for edges $e$ outside a spanning tree of $\Gamma$. Therefore the corollary follows from the classification of closed surfaces. 
	\end{proof}

	\begin{remark}\label{remark:Reeb_Theorem}
		Consider the graph $\Gamma_0$ with two vertices and one edge joining them. If $f\colon M\rightarrow \R$ is a function with finitely many critical points on a closed manifold $M$ such that $\reeb{f}$ is isomorphic to $\Gamma_0$, then $f$ has only two critical points. Reeb Theorem \cite[Theorem 1']{Milnor:Reeb} asserts that $M$ is homeomorphic to the $n$-dimensional sphere $\es^{n}$.
	\end{remark}

	It turns out that $\Gamma_0$ is the only graph which does not arise as the Reeb graph for surfaces other than the sphere. The answer to the question in Problem \ref{main_problem} is~the following theorem.

	\begin{theorem}\label{thm:reeb_graphs_of_surfaces}
		Let $\Gamma \neq \Gamma_0$ be a finite graph with good orientation and $\Sigma$ be a~closed surface. Then there exists a function $f\colon \Sigma \rightarrow \R$ with finitely many critical points such that $\reeb{f}$ is isomorphic to $\Gamma$ if and only if $\beta_1(\Gamma) \leq \reeb{\Sigma}$. If~$\Gamma = \Gamma_0$, then it can be realized only for $\Sigma = S^2$.
	\end{theorem}
	
	\begin{proof}
		The only if part is straightforward. For the reverse implication let $k:= \beta_1(\Gamma) \leq \reeb{\Sigma}$. If $\Sigma = \Sigma_g$ and $g=k$ or if $\Sigma = S_g$ and $g=2k$, then the statement follows from Corollary \ref{corollary:surface_from_theorem}.
		
		If $\Sigma = \Sigma_g$ and $g>k$, we need to change  one surface $N_v$ at vertex $v$ of degree at least $2$ in the proof of Corollary \ref{corollary:surface_from_theorem}. Such a vertex exists, if we assume that $\Gamma \neq \Gamma_0$. In the construction of $N_v$ in Lemma \ref{lemma:realization_of_neighbourhood_of_vertex_with_spheres_on_the_boundary} let us set $t=g-k \geq 1$ and let $M$ be the resulting orientable surface with a function $f$ with finitely many critical points such that $\reeb{f} \cong \Gamma$. Then from calculations as in Corollary \ref{corollary:surface_from_theorem} we have $\chi(M) = 2 - 2k - 2t = 2- 2(k+t) = 2- 2g$, so $M\cong \Sigma_g$. 
		
		For the case when $\Sigma = S_g$ is a non-orientable surface and $g>2k$ we also need to change only one manifold $N_v$ at vertex $v$ of degree at least two by setting $r=g-2k \geq 1$ in Remark \ref{remark:additional_non_orientable_1-handles}. We obtain a non-orientable surface~$M$ with a~function $f$ with finitely many critical points such that $\reeb{f} \cong \Gamma$. As in Corollaries \ref{corollary:euler_char_for_N_v} and \ref{corollary:surface_from_theorem} we get $\chi(M) = 2 -2k-(g-2k) = 2-g$, so $M\cong S_g$.
		
		If $\Gamma= \Gamma_0 = \reeb{f}$, then $\Sigma \cong\es^2 $ by Reeb Theorem and the Reeb graph of the height function on $\es^2$ is $\Gamma_0$. 
	\end{proof}

	\begin{remark}
		The degenerate critical points of the function $f$ on $\Sigma$ in the above theorem come from the vertices of degree $2$ in $\Gamma$. Therefore $\Gamma$ can always be realized as the Reeb graph of a Morse function on $\Sigma$ but only up to homeomorphism of graphs.
	\end{remark}

	\begin{theorem}\label{theorem:realization_by_Morse_function}
		Let $\Gamma \neq \Gamma_0$ be a finite graph with good orientation, $\Delta_2$ be the number of vertices of degree $2$ in $\Gamma$ and let $\Sigma$ be a closed surface of genus $g$ (orientable or not). Then there exists a Morse function $f\colon \Sigma \rightarrow \R$ such that $\reeb{f}$ is isomorphic to $\Gamma$ if and only if 
	\begin{itemize}
		\item $g \geq \beta_1(\Gamma) + \Delta_2$, when $\Sigma$ is orientable,
		\item $g \geq 2\beta_1(\Gamma) + \Delta_2$, when $\Sigma$ is non-orientable.
		\end{itemize}
	\end{theorem}

	\begin{proof}
		Let $f$ be as in the statement. Divide $\Sigma$ into surfaces $N_v$ which correspond to small, closed neighbourhoods of vertices $v\in V$ in $\Gamma$ and into surfaces $N_e \cong \es^1 \times [0,1]$ corresponding to edges $e$ in $\Gamma$. As in Corollary \ref{corollary:surface_from_theorem}
		we have $\chi(\Sigma) = \sum_{v\in V} \chi(N_v)$. Obviously, $\chi(N_v) \leq 2-\deg(v)$, where the equality holds when $N_v$ is the sphere with $\deg(v)$ open discs removed. We show that it is not the case for vertices of degree $2$.
		
		Let $w$ be a vertex of degree $2$ in $\Gamma$ and suppose $N_v$ is the sphere with two holes. The function $f_w = f|_{N_w}$ is a Morse function with at least one critical point of index $1$. This leads (after attaching two discs ($0$ and $2$ handles) to $N_w$) to a Morse function on the sphere with $k_1 \geq 1$ critical points of index $1$ and exactly two extrema. Thus $2=\chi(\es^2) = 2 - k_1 \leq 1$, a contradiction. Therefore $\chi(N_w) \leq 2 - \deg(w) -1 = -1$.
		
		If $\Sigma$ is non-orientable, then just as in Corollary \ref{corollary:surface_from_theorem} we have $$2-g = \sum_{v\in V} \chi(N_v) \leq \sum_{v\in V} (2-\deg(v)) -\Delta_2 = 2 - 2\beta_1(\Gamma) - \Delta_2,$$ hence $g \geq 2\beta_1(\Gamma) +\Delta_2$. If $\Sigma$ is orientable, then so is $N_w$ and $\chi(N_w) \leq -2$, since the Euler characteristic of a compact orientable surface with two boundary components is even. Hence $2-2g \leq \sum_{v\in V} (2-\deg(v)) -2\Delta_2 = 2 -2\beta_1(\Gamma) - 2\Delta_2$, so $g \geq \beta_1(\Gamma) + \Delta_2$.
		
		For the reverse implication, if $\Delta_2=0$, then $\reeb{\Sigma} \geq \beta_1(\Gamma)$ and the realization by a Morse function follows from Theorem \ref{thm:reeb_graphs_of_surfaces} by the above remark. Assume that $\Delta_2 \geq 1$. We need to change in Corollary \ref{corollary:surface_from_theorem} the surfaces $N_w$ for which $\deg(w) =2$.
		
		If $\Sigma$ is orientable, then we take $N_w$ to be from Lemma \ref{lemma:realization_of_neighbourhood_of_vertex_with_spheres_on_the_boundary} for $n=2$, $k_\pm =1$ and $t=1$ for all such vertices $w$ except one, for which we set $t=t_0 := g-\beta_1(\Gamma)-\Delta_2+1 \geq 1$. Since $\chi(N_w) = 2 - \deg(w) - 2t$ by Corollary \ref{corollary:euler_char_for_N_v}, the resulting closed orientable surface $M$ has $\chi(M) = 2-2\beta_1(\Gamma)-2(\Delta_2-1) -2t_0 = 2-2g$, so $M\cong \Sigma$.

		Similarily, when $\Sigma$ is non-orientable, we change $N_w$ by a surface from Remark \ref{remark:additional_non_orientable_1-handles} for $r=1$, but for one such a vertex $w$ we set $r = r_0 := g-2\beta_1(\Gamma) -\Delta_2+1 \geq 1$. Then $\chi(N_w) = 2-\deg(w) -r$. In the result we obtain a closed non-orientable surface $M$ with $\chi(M) = 2-2\beta_1(\Gamma)-(\Delta_2-1) -r_0 = 2-g$, so $M\cong \Sigma$.
	\end{proof}

	\begin{remark}
		Martinez-Alfaro et al. \cite{Morse-Bott} showed a realization theorem for simple Morse--Bott functions on orientable surfaces. They also proved that two simple Morse--Bott functions (in particular simple Morse functions) on an orientable surface are  conjugated if and only if their Reeb graphs are isomorphic and the order of vertices induced by these functions is preserved.
	\end{remark}


	\begin{remark} \label{remark:not_good_oriented_graph}
		These constructions can be extended to graphs
		with a (not necessarily good) orientation, but without oriented cycles. Such a graph $\Gamma$ may contain a~vertex $v$ of degree at least two such that $\degin(v)=0$ ($\degout(v) =0$ respectively). For such a vertex Masumoto--Saeki in \cite[Theorem 2.1. Case (c)]{Saeki} construct a surface $N_v$  and a~function $f_v\colon N_v \to \R$ such that
		\begin{itemize}	
			\setlength\itemsep{0.05em}
			\item $\chi(N_v) = 2 - \deg(v)$,
			\item $f_v$  has infinitely many critical points, but only one critical value $a$ and $f_v^{-1}(a)$ is connected,
			\item the range of $f_v$ is $[a,a+\varepsilon]$  ($[a-\varepsilon,a]$, resp.),
			\item $f_v^{-1}(a+\varepsilon)$ ($f_v^{-1}(a-\varepsilon)$, resp.) is the disjoint union of $\deg(v)$ copies of $\es^1$,
			\item $\reeb{f_v}$ is homeomorphic to a small neighbourhood of $v$ in $\Gamma$.
		\end{itemize} 
	\end{remark}
		
		\begin{proposition}
			Let $\Sigma$ be a closed surface and $\Gamma$ be an oriented graph without oriented cycles such that $\beta_1(\Gamma) \leq \reeb{\Sigma}$. Then there exists $f\colon \Sigma \to \R$ with finitely many critical values such that $\reeb{f}$ is orientation-preserving homeomorphic to~$\Gamma$.
			
			If $\Gamma$ has a vertex $w$ such that both $\degin(w)$ and $\degout(w)$ are nonzero, then $\reeb{f}$ and $\Gamma$ are in fact isomorphic.
		\end{proposition}
	
	\begin{proof}
		Adding a new vertex in the centre of an edge of $\Gamma$ gives the homeomorphic graph, so we may assume that $\Gamma$ contains a vertex $w$ such that both $\degin(w)$ and $\degout(w)$ are nonzero. The orientation of $\Gamma$ can be induced by a continuous function which is strictly monotonic on the edges (cf. \cite[Theorem 3.1.]{Sharko}). Thus the construction is the same as in the proof of Theorem \ref{thm:reeb_graphs_of_surfaces} with the exception of vertices $v$ of degree at least two such that $\degin(v)=0$ or $\degout(v) =0$ for which we use the surface and function from Remark \ref{remark:not_good_oriented_graph}.
	\end{proof}

		

	\paragraph{Acknowledgements.} The author would like to thank Prof. Wac\l{}aw Marzantowicz for suggesting the problem and for many stimulating conversations. The author expresses his gratitude to Marek Kaluba and Zbigniew B\l{}aszczyk for useful suggestions and help in editing the text of the work and he also wishes to acknowledge the anonymous reviewer for his/her detailed and helpful comments to the manuscript.

\address
	
\end{document}